\documentclass{amsproc}%
\usepackage{amsfonts}
\usepackage{amsmath}
\usepackage{amssymb}
\usepackage{graphicx}%
\theoremstyle{plain}

\newtheorem{corollary}{Corollary}

\newtheorem{definition}{Definition}

\newtheorem{lemma}{Lemma}

\newtheorem{proposition}{Proposition}
\newtheorem{remark}{Remark}

\newtheorem{theorem}{Theorem}
\numberwithin{equation}{section}
\begin{document}
\title[Riesz Kernels and Pseudodifferential Operators]{Riesz Kernels and Pseudodifferential Operators Attached to Quadratic Forms
Over $p-$adic Fields }
\author{O. Casas-S\'{a}nchez}
\address{Universidad Nacional de Colombia, Departamento de Matem\'{a}ticas \\
Ciudad Universitaria, Bogot\'{a} D.C., Colombia.}
\email{ofcasass@unal.edu.co}
\author{W. A. Z\'{u}\~{n}iga-Galindo}
\address{Centro de Investigaci\'{o}n y de Estudios Avanzados del Instituto
Polit\'{e}cnico Nacional\\
Departamento de Matem\'{a}ticas, Unidad Quer\'{e}taro\\
Libramiento Norponiente \#2000, Fracc. Real de Juriquilla. Santiago de
Quer\'{e}taro, Qro. 76230\\
M\'{e}xico.}
\email{wazuniga@math.cinvestav.edu.mx}
\thanks{The second author was partially supported by CONACYT under Grant \# 127794.}

\begin{abstract}
We study pseudodifferential equations and Riesz kernels attached to certain
quadratic forms over $p$-adic fields. We attach to an elliptic quadratic form
of dimension two or four a family of distributions depending on a complex
parameter, the Riesz kernels, and show that these distributions form an
Abelian group under convolution. This result implies the existence of
fundamental solutions for certain pseudodifferential equations like in the
classical case.

\end{abstract}
\maketitle

\section{Introduction}

This paper aims to study Riesz kernels and pseudodifferential equations
a\-ttached to quadratic forms over $p$-adic fields. The Riesz kernels are
naturally connected with several types of (pseudo) differential equations in
the Archimedean setting, see e.g. \cite{dRham}, \cite{K-V}, \cite{Rubin},
\cite{Riesz}, and non Archimedean one, see e.g. \cite{A-K-S}, \cite{Koch},
\cite{R-Zu}, \cite{R-Zu1}, \cite{Taibleson}, \cite{V-V-Z}. In particular, in
the non Archimedean setting, Riesz kernels attach to `polynomials of degree
one' has been used to solve pseudodifferential equations \cite{A-K-S},
\cite{Koch}, \cite{V-V-Z}. Our initial motivation was to extend \ these
results to the case of polynomials of higher degree to obtain $p$-adic analogs
of the results of \cite{dRham}, \cite{Riesz}. To present our results consider
the diagonal quadratic form $f\left(  \xi\right)  =a_{1}\xi_{1}^{2}%
+\ldots+a_{n}\xi_{n}^{2}$, the local zeta function attached to $f$ is the
distribution defined by%
\[
\left(  \left\vert f\right\vert _{p}^{s},\phi\right)  =%
{\displaystyle\int\limits_{\mathbb{Q}_{p}^{n}\smallsetminus f^{-1}\left(
0\right)  }}
\left\vert f\left(  \xi\right)  \right\vert _{p}^{s}\phi\left(  \xi\right)
d^{n}\xi\text{, }\operatorname{Re}\left(  s\right)  >0.
\]
These distributions, called local zeta functions, were introduced in the 50's
by I. Gel'fand and A. Weil, see \cite{G-S}, \cite{Igusa}. A Riesz kernel is a
local zeta function multiplied by a suitable gamma factor. In the cases in
which $n=2$, $4$ and the quadratic form is elliptic,\ we show that the Riesz
kernels, considered as distributions on certain $p$-adic Lizorkin spaces, form
an Abelian group under the operation of convolution, see Theorem \ref{mainA}
and Remark \ref{nota_dim_2}. The proof of this fact depends on a theorem of
Rallis-Schiffmann that asserts that the distributions of type $\left\vert
f\right\vert _{p}^{s}$ satisfy certain functional equations, see \cite{R-S},
also \cite{Igusa3}, \cite{Sato-Shi}, \cite{FSato1}, \cite{FSato2}. In order to
use this result we compute all the gamma factors that appear in the functional
equation for $p$-adic quadratic forms of type $a_{1}\xi_{1}^{2}+a_{2}\xi
_{2}^{2}+\cdots+a_{n}\xi_{n}^{2}$, see Theorem \ref{funcional}. As
consequence, we obtain \ fundamental solutions for certain\ pseudodifferential
equations, see Theorem \ref{mainB} an Remark \ref{nota_ultima}. The
fundamental solutions presented here are `classical solutions', see Definition
\ref{Classical_sol}, while those present in \cite{Koch} and \cite{Z-G1}%
-\cite{Z-G5}\ are `weak solutions'. We also obtain the existence of a
pseudodifferential operator $\boldsymbol{f}\left(  \partial,1\right)  $,
acting on a space of Lizorkin distributions, and a gamma factor $A(s)$ such
that $\boldsymbol{f}\left(  \partial,1\right)  \left\vert f\right\vert
_{p}^{s+1}=A(s)\left\vert f\right\vert _{p}^{s}$, where $f$ is an elliptic
quadratic form of dimension $2$ or $4$, see Theorem \ref{mainB} and Remark
\ref{nota_ultima}. This is a non Archimedean pseudodifferential (particular)
version of a celebrated result of Sato-Bernstein, see \cite{Sato-Shi},
\cite{Igusa}, \cite[Section 6.1.2]{Z-G5}. Thus, a natural problem is to study
the existence of pseudodifferential Sato-Bernstein operators, and the
corresponding functions, in the setting \ $p$-adic prehomogeneous vector
spaces, this problem was posed in \cite[Section 6.1.2]{Z-G5}. Finally, the
results obtained here can be applied to study other types of
pseudodifferential equations, these results will appear in a separate article elsewhere.

\textbf{Acknowledgement}. The authors want to thank to Professor Fumihiro Sato
for his kind assistance on the functional equation for the local zeta function
attached to a quadratic form. In particular, we are very grateful to him for
allowing us to use some of the ideas of his unpublished manuscript
\cite{FSato3}. The second author wants to thank to Professor Sergii Torba for
several useful comments and discussions about this article. The authors also
want to thank to the referee for his/her careful reading of the article.

\section{\label{Section1}Preliminaries}

In this section we fix the notation and collect some basic results on $p$-adic
analysis that we will use through the article. For a detailed exposition on
$p$-adic analysis the reader may consult \cite{A-K-S}, \cite{Taibleson},
\cite{V-V-Z}.

\subsection{The field of $p$-adic numbers}

Along this article $p$ will denote a prime number different from $2$. The
field of $p-$adic numbers $\mathbb{Q}_{p}$ is defined as the completion of the
field of rational numbers $\mathbb{Q}$ with respect to the $p-$adic norm
$|\cdot|_{p}$, which is defined as
\[
|x|_{p}=%
\begin{cases}
0 & \text{if }x=0\\
p^{-\gamma} & \text{if }x=p^{\gamma}\dfrac{a}{b},
\end{cases}
\]
where $a$ and $b$ are integers coprime with $p$. The integer $\gamma:=ord(x)$,
with $ord(0):=+\infty$, is called the\textit{ }$p-$\textit{adic order of} $x$.
We extend the $p-$adic norm to $\mathbb{Q}_{p}^{n}$ by taking%
\[
||x||_{p}:=\max_{1\leq i\leq n}|x_{i}|_{p},\qquad\text{for }x=(x_{1}%
,\dots,x_{n})\in\mathbb{Q}_{p}^{n}.
\]
We define $ord(x)=\min_{1\leq i\leq n}\{ord(x_{i})\}$, then $||x||_{p}%
=p^{-\text{ord}(x)}$. Any $p-$adic number $x\neq0$ has a unique expansion
$x=p^{ord(x)}\sum_{j=0}^{\infty}x_{i}p^{j}$, where $x_{j}\in\{0,1,2,\dots
,p-1\}$ and $x_{0}\neq0$. By using this expansion, we define \textit{the
fractional part of }$x\in\mathbb{Q}_{p}$, denoted $\{x\}_{p}$, as the rational
number
\[
\{x\}_{p}=%
\begin{cases}
0 & \text{if }x=0\text{ or }ord(x)\geq0\\
p^{\text{ord}(x)}\sum_{j=0}^{-ord(x)-1}x_{j}p^{j} & \text{if }ord(x)<0.
\end{cases}
\]
For $\gamma\in\mathbb{Z}$, denote by $B_{\gamma}^{n}(a)=\{x\in\mathbb{Q}%
_{p}^{n}:||x-a||_{p}\leq p^{\gamma}\}$ \textit{the ball of radius }$p^{\gamma
}$ \textit{with center at} $a=(a_{1},\dots,a_{n})\in\mathbb{Q}_{p}^{n}$, and
take $B_{\gamma}^{n}(0):=B_{\gamma}^{n}$. Note that $B_{\gamma}^{n}%
(a)=B_{\gamma}(a_{1})\times\cdots\times B_{\gamma}(a_{n})$, where $B_{\gamma
}(a_{i}):=\{x\in\mathbb{Q}_{p}:|x_{i}-a_{i}|_{p}\leq p^{\gamma}\}$ is the
one-dimensional ball of radius $p^{\gamma}$ with center at $a_{i}\in
\mathbb{Q}_{p}$. The ball $B_{0}^{n}(0)$ is equals the product of $n$ copies
of $B_{0}(0):=\mathbb{Z}_{p}$, \textit{the ring of }$p-$\textit{adic integers}.

\subsection{The Bruhat-Schwartz space}

A complex-valued function $\varphi$ defined on $\mathbb{Q}_{p}^{n}$ is
\textit{called locally constant} if for any $x\in\mathbb{Q}_{p}^{n}$ there
exist an integer $l(x)\in\mathbb{Z}$ such that%
\begin{equation}
\varphi(x+x^{\prime})=\varphi(x)\text{ for }x^{\prime}\in B_{l(x)}^{n}.
\label{local_constancy}%
\end{equation}
A function $\varphi:\mathbb{Q}_{p}^{n}\rightarrow\mathbb{C}$ is called a
\textit{Bruhat-Schwartz function (or a test function)} if it is locally
constant with compact support. The $\mathbb{C}$-vector space of
Bruhat-Schwartz functions is denoted by $\mathbf{S}(\mathbb{Q}_{p}^{n})$. For
$\varphi\in\mathbf{S}(\mathbb{Q}_{p}^{n})$, the largest of such number
$l=l(\varphi)$ satisfying (\ref{local_constancy}) is called \textit{the
exponent of local constancy of} $\varphi$.

Let $\mathbf{S}^{\prime}(\mathbb{Q}_{p}^{n})$ denote the set of all
functionals (distributions) on $\mathbf{S}(\mathbb{Q}_{p}^{n})$. All
functionals on $\mathbf{S}(\mathbb{Q}_{p}^{n})$ are continuous.

Set $\chi(y)=\exp(2\pi i\{y\}_{p})$ for $y\in\mathbb{Q}_{p}$. The map
$\chi(\cdot)$ is an additive character on $\mathbb{Q}_{p}$, i.e. a continuos
map from $\mathbb{Q}_{p}$ into $S$ (the unit circle) satisfying $\chi
(y_{0}+y_{1})=\chi(y_{0})\chi(y_{1})$, $y_{0},y_{1}\in\mathbb{Q}_{p}$.

Given $\xi=(\xi_{1},\dots,\xi_{n})$ and $x=(x_{1},\dots,x_{n})\in
\mathbb{Q}_{p}^{n}$, we set $\xi\cdot x:=\sum_{j=1}^{n}\xi_{j}x_{j}$. The
Fourier transform of $\varphi\in\mathbf{S}(\mathbb{Q}_{p}^{n})$ is defined as
\[
(\mathcal{F}\varphi)(\xi)=\int_{\mathbb{Q}_{p}^{n}}\chi(-\xi\cdot
x)\varphi(\xi)d^{n}x\quad\text{for }\xi\in\mathbb{Q}_{p}^{n},
\]
where $d^{n}x$ is the Haar measure on $\mathbb{Q}_{p}^{n}$ normalized by the
condition $vol(B_{0}^{n})=1$. The Fourier transform is a linear isomorphism
from $\mathbf{S}(\mathbb{Q}_{p}^{n})$ onto itself satisfying $(\mathcal{F}%
(\mathcal{F}\varphi))(\xi)=\varphi(-\xi)$. We will also use the notation
$\mathcal{F}_{x\rightarrow\xi}\varphi$ and $\widehat{\varphi}$\ for the
Fourier transform of $\varphi$.

\subsection{Operations on Distributions}

Let $\Omega$\ denote the characteristic function of the interval $\left[
0,1\right]  $. Then $\Delta_{k}\left(  x\right)  :=\Omega\left(
p^{-k}\left\Vert x\right\Vert _{p}\right)  $ is the characteristic function of
the ball $B_{k}^{n}\left(  0\right)  $.

\subsubsection{Convolution}

Given $f,g\in\mathbf{S}^{\prime}\left(  \mathbb{Q}_{p}^{n}\right)  $, their
convolution $f\ast g$ is defined by%
\[
\left\langle f\ast g,\varphi\right\rangle =\lim_{k\rightarrow+\infty
}\left\langle f\left(  y\right)  \times g\left(  x\right)  ,\Delta_{k}\left(
x\right)  \varphi\left(  x+y\right)  \right\rangle
\]
if the limit exists for all $\varphi\in\mathbf{S}\left(  \mathbb{Q}_{p}%
^{n}\right)  $. We recall that if $f\ast g$ exists, then $g\ast f$ exists and
$f\ast g=g\ast f$, see e.g. \cite[Section VII.1]{V-V-Z}. In the case in which
$g=\psi\in\mathbf{S}\left(  \mathbb{Q}_{p}^{n}\right)  $,
\[
f\ast\psi\left(  x\right)  =\left\langle f\left(  y\right)  ,\psi\left(
x-y\right)  \right\rangle ,
\]
see e.g. \cite[Section VII.1]{V-V-Z}.

\subsubsection{Fourier transform}

The Fourier transform $\mathcal{F}\left[  f\right]  $ of a distribution
$f\in\mathbf{S}^{\prime}\left(  \mathbb{Q}_{p}^{n}\right)  $ is defined by%
\[
\left\langle \mathcal{F}\left[  f\right]  ,\varphi\right\rangle =\left\langle
f,\mathcal{F}\left[  \varphi\right]  \right\rangle \text{ for all }\varphi
\in\mathbf{S}\left(  \mathbb{Q}_{p}^{n}\right)  \text{.}%
\]
The Fourier transform $f\rightarrow\mathcal{F}\left[  f\right]  $ is a linear
isomorphism from $\mathbf{S}^{\prime}\left(  \mathbb{Q}_{p}^{n}\right)
$\ onto $\mathbf{S}^{\prime}\left(  \mathbb{Q}_{p}^{n}\right)  $. Furthermore,
$f=\mathcal{F}\left[  \mathcal{F}\left[  f\right]  \left(  -\xi\right)
\right]  $.

\subsubsection{Multiplication}

Set $\delta_{k}\left(  x\right)  :=p^{nk}\Omega\left(  p^{k}\left\Vert
x\right\Vert _{p}\right)  $ for $k\in\mathbb{N}$. Given $f,g\in\mathbf{S}%
^{\prime}\left(  \mathbb{Q}_{p}^{n}\right)  $, their product $f\cdot g$ is
defined by%
\[
\left\langle f\cdot g,\varphi\right\rangle =\lim_{k\rightarrow+\infty
}\left\langle g,\left(  f\ast\delta_{k}\right)  \varphi\right\rangle
\]
if the limit exists for all $\varphi\in\mathbf{S}\left(  \mathbb{Q}_{p}%
^{n}\right)  $. We recall that \ the existence of the product $f\cdot g$ is
equivalent \ to the existence of $\mathcal{F}\left[  f\right]  \ast
\mathcal{F}\left[  g\right]  $. In addition, $\mathcal{F}\left[  f\cdot
g\right]  =\mathcal{F}\left[  f\right]  \ast\mathcal{F}\left[  g\right]  $ and
$\mathcal{F}\left[  f\ast g\right]  =\mathcal{F}\left[  f\right]
\cdot\mathcal{F}\left[  g\right]  $, see e.g. \cite[Section VII.5]{V-V-Z}. The
following result will be used later on.

\begin{lemma}
[{\cite[Section VII.5]{V-V-Z}.}]\label{producdistributions}Let $f,g$ functions
in $L_{loc}^{1}$ for which the function%
\[%
{\displaystyle\int\limits_{\mathbb{Q}_{p}^{n}}}
g\left(  x\right)  \varphi\left(  x\right)  f\left(  x-\xi\right)
d^{n}x\text{,}%
\]
is continuos at $\xi=0\in\mathbb{Q}_{p}^{n}$, for any\ $\varphi\in
\mathbf{S}\left(  \mathbb{Q}_{p}^{n}\right)  $. Then the product $f\cdot g$ is
in $\mathbf{S}^{\prime}\left(  \mathbb{Q}_{p}^{n}\right)  $ and the
distribution is induced by the pointwise product $f\left(  x\right)  g\left(
x\right)  $.
\end{lemma}

\subsection{The Hilbert Symbol}

The Hilbert symbol $(a,b)_{p}$, $a,b\in\mathbb{Q}_{p}^{\times}$, is defined by%
\[
(a,b)_{p}=\left\{
\begin{array}
[c]{ll}%
1 & \text{if $ax^{2}+by^{2}-z^{2}=0$ has a solution }(x,y,z)\neq\left(
0,0,0\right)  \text{ in $\mathbb{Q}_{p}^{3}$}\\
-1 & \text{otherwise.}%
\end{array}
\right.
\]
The Hilbert symbol possesses the following properties (see e.g. Theorem 3.3.1
\cite{Ki}):%

\begin{equation}
(a,b)_{p}=(b,a)_{p}\text{ and }(a,c^{2})_{p}=1\text{, for }a,b,c\in
\mathbb{Q}_{p}^{\times}; \label{HS_1}%
\end{equation}

\begin{equation}
(ab,c)_{p}=(a,c)_{p}(b,c)_{p}\text{, for }a,b,c\in\mathbb{Q}_{p}^{\times};
\label{HS_2}%
\end{equation}

\begin{equation}%
\begin{cases}
(a,b)_{p}=1 & \text{for $a,b\in\mathbb{Z}_{p}^{\times}$}\\
(a,p)_{p}=\left(  \dfrac{a_{0}}{p}\right)  & \text{for $a\in\mathbb{Z}%
_{p}^{\times},$}%
\end{cases}
\label{HS_3}%
\end{equation}
where $a_{0}\in\mathbb{Z}$, with $a\equiv a_{0}\operatorname{mod}%
\mathbb{Z}_{p}$, and $\left(  \dfrac{a_{0}}{p}\right)  $ is the Legendre symbol.

Along this article $\left[  \mathbb{Q}_{p}^{\times}\right]  ^{2}$ denotes the
subgroup of squares of $\mathbb{Q}_{p}^{\times}$. We recall that
$\mathbb{Q}_{p}^{\times}/\left[  \mathbb{Q}_{p}^{\times}\right]  ^{2}$ is a
finite group with four elements. We fix $\left\{  1,\epsilon,p,\epsilon
p\right\}  $ to be a set of representatives, here $\epsilon$\ is unit which is
not square.

It is clear that $\left(  a,b\right)  _{p}$ does not change when $a$ and $b$
are multiplied by squares, thus the Hilbert symbol gives rise a map from
$\mathbb{Q}_{p}^{\times}/\left[  \mathbb{Q}_{p}^{\times}\right]  ^{2}%
\times\mathbb{Q}_{p}^{\times}/\left[  \mathbb{Q}_{p}^{\times}\right]  ^{2}$
into $\left\{  1,-1\right\}  $. For a fixed $\beta\in\mathbb{Q}_{p}^{\times}$,
${\LARGE \pi}_{\beta}\left(  t\right)  =\left(  \beta,t\right)  _{p}$ defines
a multiplicative character on $\mathbb{Q}_{p}^{\times}$, the multiplicatively
of ${\LARGE \pi}_{\beta}$ follows from property (\ref{HS_2}).

\subsection{The Weil Constant}

Let%
\begin{equation}
f(x):=a_{1}x_{1}^{2}+a_{2}x_{2}^{2}+\cdots+a_{n}x_{n}^{2},\quad a_{i}%
\in\mathbb{Q}_{p}^{\times},\quad i=1,2,\dots,n, \label{cuadra}%
\end{equation}
be a \textit{quadratic form}. A such quadratic form is characterized by three invariants:

(i) the dimension $n$;

(ii) the discriminant $D=a_{1}a_{2}\cdots a_{n}\operatorname{mod}\left[
\mathbb{Q}_{p}^{\times}\right]  ^{2}$;

(iii) the Hasse invariant ${H=\prod_{i<j}(a_{i},a_{j})_{p}}$.

By \cite[Theoreme 2]{We2}, see also \cite[Theoreme 1.1]{R-S}, there exist a
complex constant $\gamma(f)$ of absolute value one, such that
\begin{align}
&  \int_{\mathbb{Q}_{p}^{n}}\hat{\varphi}(x)\chi(tf(x))d^{n}x\nonumber\\
&  =\gamma(tf)|t|_{p}^{-n/2}|D|_{p}^{-1/2}\int_{\mathbb{Q}_{p}^{n}}%
\varphi(x)\chi\left(  -\frac{1}{t}f{\left(  \frac{x_{1}}{2a_{1}},\dots
,\frac{x_{n}}{2a_{n}}\right)  }\right)  d^{n}x, \label{weil_fun_eq}%
\end{align}
for all $t\in\mathbb{Q}_{p}^{\times}$, where $D=a_{1}a_{2}\cdots a_{n}$.

Since $\gamma(f)=\gamma(a_{1}x_{1}^{2})\cdots\gamma(a_{n}x_{n}^{2})$, see e.g.
\cite[p. 173]{We2}, the calculation of $\gamma(f)$ is reduced to the case
$n=1$. For a $\alpha\in\mathbb{Q}_{p}^{\times}$, we set $\gamma(\alpha
):=\gamma(\alpha x_{1}^{2})$.

\begin{lemma}
\label{lemma0}For a unit $u\in\mathbb{Z}_{p}^{\times}$, with $u\equiv
u_{0}\operatorname{mod}p\mathbb{Z}_{p}$, we have $\gamma(u)=1$ and
$\gamma(up)=\left(  \frac{u_{0}}{p}\right)  \sigma_{p}$, where
\begin{equation}
{\sigma_{p}:=%
\begin{cases}
1 & \text{if }p\equiv1\operatorname{mod}4\\
\sqrt{-1} & \text{if }p\equiv3\operatorname{mod}4.
\end{cases}
} \label{sigma_p}%
\end{equation}

\end{lemma}

\begin{proof}
Take $\varphi(x)$ to be the characteristic function of $\mathbb{Z}_{p}$ and
$u\in\mathbb{Z}_{p}^{\times}$, by (\ref{weil_fun_eq})
\begin{align*}
\int_{\mathbb{Q}_{p}}\hat{\varphi}(x)\chi(ux^{2})dx  &  =\gamma(u)|u|_{p}%
^{-1/2}\int_{\mathbb{Q}_{p}}\varphi(x)\chi\left(  -\frac{x^{2}}{4u}\right)
dx\\
\int_{\mathbb{Z}_{p}}dx  &  =\gamma(u).
\end{align*}
In the case $up$ with $u\in\mathbb{Z}_{p}^{\times}$, by applying
(\ref{weil_fun_eq}) we have%
\begin{align*}
\int_{\mathbb{Q}_{p}}\hat{\varphi}(x)\chi(upx^{2})dx  &  =\gamma
(up)|up|_{p}^{-1/2}\int_{\mathbb{Q}_{p}}\varphi(x)\chi\left(  -\frac{x^{2}%
}{4pu}\right)  dx\\
\int_{\mathbb{Z}_{p}}\chi(upx^{2})dx  &  =\gamma(up)p^{1/2}\int_{\mathbb{Z}%
_{p}}\chi\left(  -\frac{x^{2}}{4pu}\right)  dx\\
1  &  =\gamma(up)p^{1/2}\int_{\mathbb{Z}_{p}}\chi\left(  -\frac{x^{2}}%
{4pu}\right)  dx.
\end{align*}

If $z\in\mathbb{Z}_{p}\smallsetminus\left\{  0\right\}  $ we set
$z=z_{0}+z_{1}p+\ldots+z_{k}p^{k}+\ldots$ with $z_{k}=\left\{  0,1,\ldots
,p-1\right\}  $. Now by changing variables ($x=2uy$) in the previous
integral:
\begin{align*}
\int_{\mathbb{Z}_{p}}\chi\left(  -\frac{x^{2}}{4pu}\right)  dx  &
=\int_{|y|_{p}\leq1}\chi\left(  -\frac{uy^{2}}{p}\right)  dy\\
&  =\int_{|y|_{p}=1}\chi\left(  -\frac{uy^{2}}{p}\right)  dy+\int_{|y|_{p}%
<1}\chi\left(  -\frac{uy^{2}}{p}\right)  dy\\
&  =\frac{1}{p}\sum_{y_{0}=1}^{p-1}exp\left\{  -2\pi i\frac{u_{0}y_{0}^{2}}%
{p}\right\}  +\frac{1}{p}=\frac{1}{p}\sum_{y_{0}=0}^{p-1}exp\left\{  -2\pi
i\frac{u_{0}y_{0}^{2}}{p}\right\} \\
&  =p^{-1/2}\left(  \frac{-u_{0}}{p}\right)  \sigma_{p},
\end{align*}
where in the last step we used a result of Gauss on quadratic exponential
sums, see e.g. \cite[p. 55]{V-V-Z}. Therefore
\[
\gamma(up)=\frac{1}{\left(  \frac{-1}{p}\right)  \left(  \frac{u_{0}}%
{p}\right)  \sigma_{p}}=\left(  \frac{u_{0}}{p}\right)  \sigma_{p}.
\]

\end{proof}

The next lemma shows the relation between the constant $\gamma$ and the
Hilbert symbol.

\begin{lemma}
\label{lemma1}With the above notation, the following assertions hold.

(i) $\gamma(-a)\gamma(a)=1$.

(ii) Set $h(x)=x_{1}^{2}-ax_{2}^{2}-bx_{3}^{2}+abx_{4}^{2}$ with
$a,b\in\mathbb{Q}_{p}^{\times}$. Then
\[
\gamma(h)=\gamma(1)\gamma(-a)\gamma(-b)\gamma(ab)=(a,b)_{p}.
\]

(iii) If $n\equiv0\operatorname{mod}2$, then $\gamma(tf)=\gamma(f)(t,D^{\ast
})_{p}$ for any $t\in\mathbb{Q}_{p}^{\times}$, where
\[
D^{\ast}:=(-1)^{\frac{n}{2}}D.
\]

\end{lemma}

\begin{proof}
(i) See \cite[Section No. 25, p. 173 ]{We2}. (ii) See \cite[Section No. 28, p.
176]{We2}. (iii) See \cite[Proposition 1.7]{R-S}.
\end{proof}

\subsection{Local zeta functions}

For $a>0$ and $s\in\mathbb{C}$ we set $a^{s}:=e^{s\ln a}$. Let $f(x)$ be a
quadratic form over $\mathbb{Q}_{p}$ and ${\LARGE \pi}_{\beta}(t):=(\beta
,t)_{p},\quad t\in\mathbb{Q}_{p}^{\times}$ as before. The local function zeta
attached to $\left(  f,{\LARGE \pi}_{\beta}\right)  $ is the distribution
given by
\begin{equation}
Z_{\varphi}(s,{\LARGE \pi}_{\beta},f):=Z_{\varphi}(s,{\LARGE \pi}_{\beta
})=\int_{\mathbb{Q}_{p}^{n}\smallsetminus f^{-1}(0)}{\LARGE \pi}_{\beta
}(f(x))|f(x)|_{p}^{s-n/2}\varphi(x)d^{n}x\text{, \ }\label{zeatfunction}%
\end{equation}
$\varphi\in\mathbf{S}\left(  \mathbb{Q}_{p}^{n}\right)  $\ and
$\operatorname{Re}(s)>\frac{n}{2}$. If $\beta=1$ we use $Z_{\varphi}(s,f)$
instead of $Z_{\varphi}(s,{\LARGE \pi}_{1},f)$. The local zeta functions are
defined for arbitrary polynomials and arbitrary multiplicative characters.
These objects were introduced in the 60's by A. Weil and since then they have
been studied intensively, see e.g. \cite{Igusa}. The local zeta function
$Z_{\varphi}(s,{\LARGE \pi}_{\beta})$ is a distribution on $\mathbf{S}\left(
\mathbb{Q}_{p}^{n}\right)  $\ for $\operatorname{Re}(s)>\frac{n}{2}$, which
admits a meromorphic continuation to the whole complex plane (for arbitrary
$f$ and ${\LARGE \pi}_{\beta}$) such that $Z_{\varphi}(s,{\LARGE \pi}_{\beta
})$ is a rational function of $p^{-s}$, see \cite[Theorem 8.2.1]{Igusa}.

\subsection{Functional equations}

It is well-known that the Fourier transform of the distribution ${\LARGE \pi
}_{\beta}(t)|t|_{p}^{s-1}$ is $\rho({\LARGE \pi}_{\beta},s){\LARGE \pi}%
_{\beta}^{-1}(t)|t|_{p}^{-s}$ i.e.
\begin{equation}
\int_{\mathbb{Q}_{p}^{\times}}\widehat{\varphi}(t){\LARGE \pi}_{\beta
}(t)|t|_{p}^{s-1}dt=\rho({\LARGE \pi}_{\beta},s)\int_{\mathbb{Q}_{p}^{\times}%
}\varphi(t){\LARGE \pi}_{\beta}^{-1}(t)|t|_{p}^{-s}dt, \label{tate1}%
\end{equation}
for all $\varphi(t)\in S(\mathbb{Q}_{p})$, see e.g. \cite[Section
VIII.2]{V-V-Z}. We recall that (\ref{tate1}) is a particular case of the
functional equation for the Iwasawa-Tate local zeta function see e.g.
\cite[Theoerem 2.4.1 and Lemma 2.4.3]{Tate}.

We now compute the factors $\rho(\pi_{\beta},s)$ appearing in (\ref{tate1}).

\begin{lemma}
\label{funro} Set $\mathbb{Q}_{p}^{\times}/\left[  \mathbb{Q}_{p}^{\times
}\right]  ^{2}=$ $\left\{  1,\epsilon,p,\epsilon p\right\}  $ where $\epsilon
$\ is unit which is not square. Then

(i) ${\rho(}{\LARGE \pi}{_{1},}${$s$}${)=\frac{1-p^{s-1}}{1-p^{-s}}};$

(ii) ${\rho(}{\LARGE \pi}{_{\epsilon},}${$s$}${)=\frac{1+p^{s-1}}{1+p^{-s}}}$;

(iii) ${\rho(}{\LARGE \pi}{_{\eta},}${$s$}${)=\pm\sigma_{p}p^{s-\frac{1}{2}}}%
$, $\eta=p,\epsilon p$ with ${\sigma_{p}}$\ as in (\ref{sigma_p}).
\end{lemma}

\begin{proof}
(i) Take $\varphi(t)$ to be the characteristic function of $\mathbb{Z}_{p}$ in
(\ref{tate1}), then%

\[
\rho(\pi_{1},s)=\frac{\int_{\mathbb{Z}_{p}\smallsetminus\left\{  0\right\}
}|t|_{p}^{s-1}dt}{\int_{\mathbb{Z}_{p}\smallsetminus\left\{  0\right\}
}|t|_{p}^{-s}dt}=\frac{1-p^{s-1}}{1-p^{-s}}.
\]

(ii) Note that ${\LARGE \pi}_{\epsilon}(t)=(-1)^{ord(t)}$, see \cite[Lemma on
p. 130]{V-V-Z}, by taking $\varphi(t)$ to be the characteristic function of
$\mathbb{Z}_{p}$ in (\ref{tate1}), we have
\[
\rho({\LARGE \pi}_{\epsilon},s)=\frac{1+p^{s-1}}{1+p^{-s}}.
\]

(iii) Set
\[
\mathbb{Q}_{p,\eta}^{\times}:=\left\{  x\in\mathbb{Q}_{p}^{\times}\mid
x=a^{2}-\eta b^{2},\quad a,b\in\mathbb{Q}_{p}\right\}
\]

and
\[
sgn_{\eta}(x):=%
\begin{cases}
1 & \text{if $x\in\mathbb{Q}_{p,\eta}^{\times}$}\\
-1 & \text{if $x\notin\mathbb{Q}_{p,\eta}^{\times}.$}%
\end{cases}
\]

In \cite[p. 129]{V-V-Z} is proved that$\ {\rho(}{\LARGE \pi}{_{\eta},}${$s$%
}${)=\pm\sqrt{sgn_{\eta}(-1)}p^{s-\frac{1}{2}}}$ for $\eta=p,\epsilon p$.
Since $(t,\eta)_{p}=sgn_{\eta}(t)$ we have
\[
{\pm\sqrt{sgn_{\eta}(-1)}=\pm\sqrt{(\eta,-1)_{p}}=\pm\sqrt{\left(  \frac
{-1}{p}\right)  }=\pm\sigma_{p}.}%
\]

\end{proof}

Set ${f^{\ast}(x):=f\left(  \frac{x_{1}}{a_{1}},\dots,\frac{x_{n}}{a_{n}%
}\right)  }${, and}
\[
Z_{\varphi}^{\ast}(s,{\LARGE \pi}_{\beta}):=\int_{\mathbb{Q}_{p}%
^{n}\smallsetminus f^{\ast-1}(0)}{\LARGE \pi}_{\beta}(f^{\ast}(x))|f^{\ast
}(x)|_{p}^{s-n/2}\varphi(x)d^{n}x.
\]

\begin{theorem}
[{\cite[Theoreme 22-13]{R-S}}]\label{funcional} If $n\equiv0\operatorname{mod}%
2$, then $Z_{\varphi}(s)$ satisfies
\[
Z_{\hat{\varphi}}(s)=\rho({\LARGE \pi}_{1},s-\frac{n}{2}+1)\rho({\LARGE \pi
}_{D^{\ast}},s)|D|_{p}^{-1/2}\gamma(f)Z_{\varphi}^{\ast}(-s+n/2,{\LARGE \pi
}_{D^{\ast}})
\]
for any $\varphi\in\mathbf{S}\left(  \mathbb{Q}_{p}^{n}\right)  $.
\end{theorem}

\begin{proof}
The announced formula is a particular case of formula 2-20 in \cite{R-S}. We
note that our functional equation equals up to a constant to the functional
equation in \cite{R-S}, this is due fact that we used a different
normalization for the Haar measure.
\end{proof}

\subsection{Some explicit functional equations}

\begin{corollary}
\label{Cor_Fun_Eq}Let $f(x)$ be as before. Assume that $n\equiv
0\operatorname{mod}2$ and that $D^{\ast}$ is a square. Then%
\[
Z_{\hat{\varphi}}(s)=\rho({\LARGE \pi}_{1},s-\frac{n}{2}+1)\rho({\LARGE \pi
}_{1},s)|D|_{p}^{-1/2}\gamma(f)Z_{\varphi}^{\ast}(-s+n/2)
\]
for any $\varphi\in\mathbf{S}\left(  \mathbb{Q}_{p}^{n}\right)  $.
\end{corollary}

\begin{proposition}
\label{eqfun2} If $f(x)=x_{1}^{2}-\eta x_{2}^{2}$, $\eta=\epsilon,p,p\epsilon
$, then
\begin{align*}
&
{\displaystyle\int\limits_{\mathbb{Q}_{p}^{2}\smallsetminus\left\{  0\right\}
}}
|f(x)|_{p}^{s-1}\widehat{\varphi}(x)d^{2}x\\
&  =%
\begin{cases}
\dfrac{1-p^{2(s-1)}}{1-p^{-2s}}%
{\displaystyle\int\limits_{\mathbb{Q}_{p}^{2}\smallsetminus\left\{  0\right\}
}}
{|\eta x_{1}^{2}-x_{2}^{2}|_{p}^{-s}\varphi(x)d}^{2}x & \text{if }%
\eta=\epsilon\\
\dfrac{1-p^{s-1}}{1-p^{-s}}%
{\displaystyle\int\limits_{\mathbb{Q}_{p}^{2}\smallsetminus\left\{  0\right\}
}}
{|\eta x_{1}^{2}-x_{2}^{2}|_{p}^{-s}\varphi(x)d}^{2}x & \text{if }%
\eta=p,p\epsilon.
\end{cases}
\end{align*}

\end{proposition}

\begin{proof}
Since $D^{\ast}=-D=\eta$ and ${\LARGE \pi}_{-D}(f^{\ast}(x))=(\eta,x_{1}%
^{2}-\eta x_{2}^{2})_{p}=1$. By Theorem \ref{funcional}, we have
\[
Z_{\widehat{\varphi}}(s)=\rho({\LARGE \pi}_{1},s)\rho({\LARGE \pi}%
_{-D},s)|D|_{p}^{-1/2}\gamma(f)Z_{\varphi}^{\ast}(1-s,{\LARGE \pi}_{1})
\]%
\[
=\rho({\LARGE \pi}_{1},s)\rho({\LARGE \pi}_{-D},s)|D|_{p}^{-1/2}\gamma
(f)|\eta|_{p}^{s}%
{\displaystyle\int\limits_{\mathbb{Q}_{p}^{2}\smallsetminus\left\{  0\right\}
}}
{|\eta x_{1}^{2}-x_{2}^{2}|_{p}^{-s}\varphi\left(  x\right)  dx}_{{1}}%
{dx}_{{2}}.
\]

The announced functional equations follow from the following calculations. (i)
Take $\eta=\epsilon$, then $|\epsilon|_{p}^{-1/2}=|\epsilon|_{p}^{s}%
=1,\quad\gamma(f)=\gamma(1)\gamma(-\epsilon)=1$, see Lemma \ref{lemma0}, and
${\LARGE \pi}_{-D}\left(  \epsilon\right)  =\left(  \epsilon,\epsilon\right)
_{p}=1$, see (\ref{HS_3}). Furthermore ${\rho(}{\LARGE \pi}{_{1},}${$s$}%
${){=}\frac{1-p^{s-1}}{1-p^{-s}}}$, ${{\rho(}{\LARGE \pi}{_{\epsilon},s)=}%
}\frac{1+p^{s-1}}{1+p^{-s}}$, see Lemma \ref{funro}.

(ii) Take $\eta=p$, $p\epsilon$, in this case we have $|\eta|_{p}%
^{-1/2}=p^{1/2}$, $\gamma(f)=\gamma(-\eta)=\frac{1}{\pm\sigma_{p}}$ (see Lemma
\ref{lemma0}) and ${\rho(}{\LARGE \pi}{_{\eta},s)=\pm\sigma_{p}p^{s-\frac
{1}{2}}}$ (see Lemma \ref{funro}). Then
\begin{align*}
\rho({\LARGE \pi}_{1},s)\rho({\LARGE \pi}_{-D},s)|D|_{p}^{-1/2}\gamma
(f)|\eta|_{p}^{s} &  ={\frac{1-p^{s-1}}{1-p^{-s}}}(\pm\sigma_{p})p^{s-\frac
{1}{2}}p^{1/2}(\frac{1}{\pm\sigma_{p}})p^{-s}\\
&  =\frac{1-p^{s-1}}{1-p^{-s}}.
\end{align*}

\end{proof}

\begin{proposition}
\label{eqfun4} Take $f(x)=x_{1}^{2}-ax_{2}^{2}-px_{3}^{2}+apx_{4}^{2}$, with
$a\in\mathbb{Z}$ a quadratic non-residue module $p$. Then
\[
\int\limits_{\mathbb{Q}_{p}^{4}\setminus\{0\}}|f(x)|^{s-2}\widehat{\varphi
}(x)d^{4}x=\frac{1-p^{s-2}}{\left(  1-p^{-s}\right)  }\int\limits_{\mathbb{Q}%
_{p}^{4}\setminus\{0\}}|apx_{1}^{2}-px_{2}^{2}-ax_{3}^{2}+x_{4}^{2}|_{p}%
^{-s}\varphi(x)d^{4}x.
\]

\end{proposition}

\begin{proof}
In this case $n=4$, $D=p^{2}a^{2}$, $D^{\ast}=D$ and $\gamma(f)=(a,p)_{p}=-1$,
see Lemma \ref{lemma1} (ii) and (\ref{HS_3}), the functional equation takes
the form
\begin{align*}
&  \int\limits_{\mathbb{Q}_{p}^{4}\setminus\{0\}}|f(x)|_{p}^{s-2}%
\widehat{\varphi}(x)d^{4}x\\
&  =-\rho({\LARGE \pi}_{1},s-1)\rho({\LARGE \pi}_{1},s)p\int
\limits_{\mathbb{Q}_{p}^{4}\setminus\{0\}}|x_{1}^{2}-a^{-1}x_{2}^{2}%
-p^{-1}x_{3}^{2}+(ap)^{-1}x_{4}^{2}|_{p}^{-s}\varphi(x)d^{4}x\\
&  =-|ap|_{p}^{s}p\frac{1-p^{s-2}}{1-p^{1-s}}\frac{1-p^{s-1}}{1-p^{-s}}%
\int\limits_{\mathbb{Q}_{p}^{4}\setminus\{0\}}|apx_{1}^{2}-px_{2}^{2}%
-ax_{3}^{2}+x_{4}^{2}|_{p}^{-s}\varphi(x)d^{4}x\\
&  =\frac{1-p^{s-2}}{\left(  1-p^{-s}\right)  }\int\limits_{\mathbb{Q}_{p}%
^{4}\setminus\{0\}}|apx_{1}^{2}-px_{2}^{2}-ax_{3}^{2}+x_{4}^{2}|_{p}%
^{-s}\varphi(x)d^{4}x.
\end{align*}

\end{proof}

\section{\label{Section2}Riesz Kernels and Lizorkin Spaces of Second Kind}

In this section we introduce a new type of Riesz kernels depending on a
complex parameter and a certain quadratic form. The main result of this
section establishes that these kernels considered as distributions on a
Lizorkin spaces of second kind form an Abelian group under convolution.

\subsection{Riesz Kernels}

In this section $f(x):=x_{1}^{2}-ax_{2}^{2}-px_{3}^{2}+apx_{4}^{2}$ with
$a\in\mathbb{Z}$ a quadratic non-residue module $p$. Note that $a\in
\mathbb{Z}^{\times}_{p}$ and that $f(x)$ is an elliptic quadratic form, i.e.
$f(x)=0\Leftrightarrow x=0$. We call the function
\[
K_{\alpha}(x):=\dfrac{1-p^{-\alpha}}{1-p^{\alpha-2}}|f(x)|_{p}^{\alpha
-2},\quad\operatorname{Re}(\alpha)>0,\ \alpha\neq2+\frac{2\pi\sqrt{-1}}{\ln
p}\mathbb{Z},
\]
\textit{the Riesz kernel attached to} $f(x)$.

\begin{lemma}
\label{lemma2}Set $\varphi$ to be the characteristic function of the ball
$\widetilde{x}_{0}+\left(  p^{m}\mathbb{Z}_{p}\right)  ^{4}$. Then%

\[
Z_{\varphi}\left(  \alpha,f\right)  =\left\{
\begin{array}
[c]{lll}%
\frac{p^{-2\alpha m}\left(  1-p^{-2}\right)  }{1-p^{-\alpha}} & \text{if} &
\widetilde{x}_{0}\in\left(  p^{m}\mathbb{Z}_{p}\right)  ^{4}\\
&  & \\%
\begin{array}
[c]{c}%
\text{holomorphic function}\\
\text{in }\alpha\text{, for }\alpha\in\mathbb{C}%
\end{array}
& \text{if} & \widetilde{x}_{0}\notin\left(  p^{m}\mathbb{Z}_{p}\right)  ^{4}.
\end{array}
\right.
\]

\end{lemma}

\begin{proof}
We consider first the case $\widetilde{x}_{0}\in\left(  p^{m}\mathbb{Z}%
_{p}\right)  ^{4}$. Set
\[
Z(\alpha):=%
{\displaystyle\int\limits_{\mathbb{Z}_{p}^{4}\smallsetminus\{0\}}}
\left\vert f(x)\right\vert _{p}^{\alpha-2}d^{4}x\text{ for }\operatorname{Re}%
(\alpha)>2.
\]
By a change of variables $Z_{\varphi}\left(  \alpha,f\right)  =p^{-2\alpha
m}Z(\alpha)$. The result follows from the following formula:
\begin{equation}
Z(\alpha)=\frac{1-p^{-2}}{1-p^{-\alpha}}\text{ for }\operatorname{Re}%
(\alpha)>2. \label{formula}%
\end{equation}

Set $\mathbb{Z}_{p}^{4}=\left(  p\mathbb{Z}_{p}\right)  ^{4}%
{\textstyle\bigsqcup}
U$ with $U=\left\{  x\in\mathbb{Z}_{p}^{4}:\left\Vert x\right\Vert
_{p}=1\right\}  $. Then%
\begin{align*}
Z(\alpha)  &  =%
{\displaystyle\int\limits_{\left(  p\mathbb{Z}_{p}\right)  ^{4}}}
\left\vert f(x)\right\vert _{p}^{\alpha-2}d^{4}x+%
{\displaystyle\int\limits_{U}}
\left\vert f(x)\right\vert _{p}^{\alpha-2}d^{4}x\\
&  =p^{-2\alpha}Z(\alpha)+%
{\displaystyle\int\limits_{U}}
\left\vert f(x)\right\vert _{p}^{\alpha-2}d^{4}x,
\end{align*}
i.e.
\[
Z(\alpha)=\frac{1}{1-p^{-2\alpha}}%
{\displaystyle\int\limits_{U}}
\left\vert f(x)\right\vert _{p}^{\alpha-2}d^{4}x\text{.}%
\]
In order to show (\ref{formula}), it is sufficient to prove the following
formula:
\begin{equation}%
{\displaystyle\int\limits_{U}}
\left\vert f(x)\right\vert _{p}^{\alpha-2}d^{4}x=(1-p^{-2})\left(
1+p^{-\alpha}\right)  \text{ for }\alpha\in\mathbb{C}. \label{Int_unidades}%
\end{equation}
This formula can be established as follows. For $i=\left(  i_{1},i_{2}%
,i_{3},i_{4}\right)  \in\left\{  0,1\right\}  ^{4}\smallsetminus\left\{
\left(  1,1,1,1\right)  \right\}  $ we define
\begin{align*}
U^{\left(  i\right)  }  &  =U_{1}^{\left(  i\right)  }\times U_{2}^{\left(
i\right)  }\times U_{3}^{\left(  i\right)  }\times U_{4}^{\left(  i\right)
},\\
U_{j}^{\left(  i\right)  }  &  :=\left\{
\begin{array}
[c]{ccc}%
p\mathbb{Z}_{p} & \text{if} & i_{j}=1\\
&  & \\
\mathbb{Z}_{p}^{\times} & \text{if} & i_{j}=0.
\end{array}
\right.
\end{align*}
Then $U=%
{\textstyle\bigsqcup\nolimits_{i}}
U^{\left(  i\right)  }$ and
\[%
{\displaystyle\int\limits_{U}}
\left\vert f(x)\right\vert _{p}^{\alpha-2}d^{4}x=%
{\displaystyle\sum\limits_{i}}
{\displaystyle\int\limits_{U^{\left(  i\right)  }}}
\left\vert f(x)\right\vert _{p}^{\alpha-2}d^{4}x:=%
{\displaystyle\sum\limits_{i}}
Z_{i}(\alpha).
\]

By a direct calculation one finds:%
\[%
\begin{tabular}
[c]{|l|l|}\hline
Index $i$ & $Z_{i}(\alpha)$\\\hline
$%
\begin{array}
[c]{c}%
\left(  1,1,1,0\right)
\end{array}
$ & $p^{-\alpha-1}\left(  1-p^{-1}\right)  $\\\hline
$%
\begin{array}
[c]{c}%
\left(  1,1,0,1\right)
\end{array}
$ & $p^{-\alpha-1}\left(  1-p^{-1}\right)  $\\\hline
$%
\begin{array}
[c]{c}%
\left(  1,1,0,0\right)
\end{array}
$ & $p^{-\alpha}\left(  1-p^{-1}\right)  ^{2}$\\\hline
$%
\begin{array}
[c]{cc}%
\left(  1,0,1,1\right)  , & \left(  0,1,1,1\right)
\end{array}
$ & $\left(  1-p^{-1}\right)  p^{-3}$\\\hline
$%
\begin{array}
[c]{ccc}%
\left(  1,0,1,0\right)  , & \left(  1,0,0,1\right)  , & \left(  0,1,1,0\right)
\\
\left(  0,1,0,1\right)  , & \left(  0,0,1,1\right)  &
\end{array}
$ & $\left(  1-p^{-1}\right)  ^{2}p^{-2}$\\\hline
$%
\begin{array}
[c]{ccc}%
\left(  1,0,0,0\right)  , & \left(  0,1,0,0\right)  , & \left(  0,0,1,0\right)
\\
\left(  0,0,0,1\right)  &  &
\end{array}
$ & $\left(  1-p^{-1}\right)  ^{3}p^{-1}$\\\hline
$%
\begin{array}
[c]{c}%
\left(  0,0,0,0\right)
\end{array}
$ & $\left(  1-p^{-1}\right)  ^{4}$\\\hline
\end{tabular}
\ \ \ \ \
\]

In the case $\widetilde{x}_{0}\notin\left(  p^{m}\mathbb{Z}_{p}\right)  ^{4}$,
$f$\ does not vanish on the ball $\widetilde{x}_{0}+\left(  p^{m}%
\mathbb{Z}_{p}\right)  ^{4}$ which implies that $Z_{\varphi}\left(
\alpha\right)  $ is a holomorphic function on the whole complex plane.
\end{proof}

\begin{lemma}
\label{lemma3}$K_{\alpha}(x)$ possesses, as a distribution on $S(\mathbb{Q}%
_{p}^{4})$, a meromorphic continuation to all $\alpha\neq2+\frac{2\pi\sqrt
{-1}}{\ln p}\mathbb{Z}$ given by
\begin{multline*}
\left\langle K_{\alpha},\varphi\right\rangle =\varphi(0)\dfrac{1-p^{-2}%
}{1-p^{\alpha-2}}+\\
\dfrac{1-p^{-\alpha}}{1-p^{\alpha-2}}\left[  \int_{||x||_{p}>1}\varphi
(x)|f(x)|_{p}^{\alpha-2}d^{4}x+\int_{||x||_{p}\leq1}\left(  \varphi
(x)-\varphi\left(  0\right)  \right)  |f(x)|_{p}^{\alpha-2}d^{4}x\right]  .
\end{multline*}

\end{lemma}

\begin{proof}
The result follows from Lemma \ref{lemma2}\ \ by%
\begin{multline*}
\left\langle K_{\alpha},\varphi\right\rangle =\dfrac{1-p^{-\alpha}%
}{1-p^{\alpha-2}}\int\limits_{\mathbb{Q}_{p}^{4}\setminus\{0\}}\varphi
(x)|f\left(  x\right)  |_{p}^{\alpha-2}d^{4}x=\\
\dfrac{1-p^{-\alpha}}{1-p^{\alpha-2}}\left[  \int_{||x||_{p}>1}\varphi
(x)|f\left(  x\right)  |_{p}^{\alpha-2}d^{4}x+\int_{||x||_{p}\leq1}\left(
\varphi(x)-\varphi\left(  0\right)  \right)  |f(x)|_{p}^{\alpha-2}%
d^{4}x\right] \\
+\varphi(0)\dfrac{1-p^{-2}}{1-p^{\alpha-2}}.
\end{multline*}

\end{proof}

From Lemma \ref{lemma3} follows that the distribution $K_{\alpha}$ has simple
poles at the points $\alpha=2+\alpha_{k}$ with $\alpha_{k}:=\frac{2k\pi
\sqrt{-1}}{\ln p},\quad k\in\mathbb{Z}$, and
\begin{equation}
{\lim_{\alpha\rightarrow\alpha_{k}}K_{\alpha}:=K_{\alpha_{k}}=\delta,}
\label{deltacero}%
\end{equation}
where ${\delta}$ denotes the Dirac distribution.

\begin{lemma}
\label{lemma4}%
\[
\int_{||x||_{p}>1}\frac{1}{|f\left(  x\right)  |_{p}^{\alpha+2}}%
dx=\frac{p^{-2\alpha}(1-p^{-2})\left(  1+p^{\alpha}\right)  }{1-p^{-2\alpha}%
},\quad\text{Re}(\alpha)>0.
\]

\end{lemma}

\begin{proof}
Set $U=(\mathbb{Z}_{p})^{4}\setminus(p\mathbb{Z}_{p})^{4}$ as before. Then
\begin{multline*}
\int_{||x||_{p}>1}\frac{1}{|f\left(  x\right)  |_{p}^{\alpha+2}}d^{4}%
x=\sum_{m=1}^{\infty}\int_{p^{-m}U}\frac{1}{|f\left(  x\right)  |^{\alpha+2}%
}d^{4}x\\
=\sum_{m=1}^{\infty}p^{-2m\alpha}\int_{U}\frac{1}{|f\left(  x\right)
|_{p}^{\alpha+2}}d^{4}x=\frac{p^{-2\alpha}}{1-p^{-2\alpha}}\int_{U}\frac
{1}{|f\left(  x\right)  |_{p}^{\alpha+2}}d^{4}x\\
=\frac{p^{-2\alpha}(1-p^{-2})\left(  1+p^{\alpha}\right)  }{1-p^{-2\alpha}},
\end{multline*}
where we used (\ref{Int_unidades}).
\end{proof}

\begin{proposition}
\label{Prop1}For $\text{Re}(\alpha)>0$ and $\varphi\in\mathbf{S}\left(
\mathbb{Q}_{p}^{4}\right)  $, the following formulas hold:

(i) $\left\langle K_{\alpha},\varphi\right\rangle =\dfrac{1-p^{-\alpha}%
}{1-p^{\alpha-2}}%
{\displaystyle\int\nolimits_{\mathbb{Q}_{p}^{4}\setminus\{0\}}}
|f\left(  x\right)  |_{p}^{\alpha-2}\varphi(x)d^{4}x,\quad\alpha\neq
2+\alpha_{k}$;

(ii) $\left\langle K_{-\alpha},\varphi\right\rangle =\dfrac{1-p^{\alpha}%
}{1-p^{-\alpha-2}}%
{\displaystyle\int\nolimits_{\mathbb{Q}_{p}^{4}}}
\frac{\varphi(x)-\varphi(0)}{|f\left(  x\right)  |_{p}^{\alpha+2}}d^{4}x$;

(iii) $(K_{\alpha}\ast\varphi)(x)=\dfrac{1-p^{-\alpha}}{1-p^{\alpha-2}}%
{\displaystyle\int\nolimits_{\mathbb{Q}_{p}^{4}\setminus\{0\}}}
|f(y)|_{p}^{\alpha-2}\varphi(x+y)d^{4}y,\quad\alpha\neq2+\alpha_{k}$;

(iv) $(K_{-\alpha}\ast\varphi)(x)=\dfrac{1-p^{\alpha}}{1-p^{-\alpha-2}}%
{\displaystyle\int\nolimits_{\mathbb{Q}_{p}^{4}}}
\frac{\varphi(x+y)-\varphi(x)}{|f(y)|_{p}^{\alpha+2}}d^{4}y.$
\end{proposition}

\begin{proof}
(i) Since every test function can be written a finite sums of characteristic
functions of balls, Lemma \ref{lemma2} implies that
\[
\dfrac{1-p^{-\alpha}}{1-p^{\alpha-2}}%
{\displaystyle\int\limits_{\mathbb{Q}_{p}^{4}\setminus\{0\}}}
|f\left(  x\right)  |_{p}^{\alpha-2}\varphi(x)d^{4}x
\]
is well-defined for $\operatorname{Re}(\alpha)>0$ and $\alpha\neq2+\alpha_{k}%
$. The announced formula follows by a calculation similar to the one done in
the proof of Lemma \ref{lemma3}.

(ii) We first note that the integral%
\[
\dfrac{1-p^{\alpha}}{1-p^{-\alpha-2}}%
{\displaystyle\int\nolimits_{\mathbb{Q}_{p}^{4}}}
\frac{\varphi(x)-\varphi(0)}{|f\left(  x\right)  |_{p}^{\alpha+2}}d^{4}x
\]
converges on $\operatorname{Re}(\alpha)>0$. Indeed, since $f\left(  x\right)
$ is an elliptic quadratic form we have%
\begin{equation}
B\left\Vert x\right\Vert _{p}^{2}\leq|f\left(  x\right)  |_{p}\leq A\left\Vert
x\right\Vert _{p}^{2}\text{ for any }x\in\mathbb{Q}_{p}^{n}%
,\label{Zuniga_ineq}%
\end{equation}
where $A$, $B$ are positive constants, cf. \cite[Lemma 1]{Z-G3}, then
\[%
{\displaystyle\int\nolimits_{\mathbb{Q}_{p}^{4}}}
\frac{\left\vert \varphi(x)-\varphi(0)\right\vert }{|f\left(  x\right)
|_{p}^{\operatorname{Re}\left(  \alpha\right)  +2}}d^{4}x\leq\frac{2\left\Vert
\varphi\right\Vert _{L^{\infty}}}{B^{\operatorname{Re}(\alpha)+2}}%
{\displaystyle\int\nolimits_{\left\Vert x\right\Vert _{p}>p^{m}}}
\frac{1}{\left\Vert x\right\Vert _{p}^{2\operatorname{Re}(\alpha)+4}}%
d^{4}x<\infty,
\]
where $m$ is the exponent of local constancy of $\varphi$.

Now%

\begin{multline*}
\dfrac{1-p^{\alpha}}{1-p^{-\alpha-2}}%
{\displaystyle\int\nolimits_{\mathbb{Q}_{p}^{4}}}
\frac{\varphi(x)-\varphi(0)}{|f\left(  x\right)  |_{p}^{\alpha+2}}d^{4}x=\\
\dfrac{1-p^{\alpha}}{1-p^{-\alpha-2}}\left\{  \int_{||x||_{p}\leq1}%
\frac{\varphi(x)-\varphi(0)}{|f\left(  x\right)  |_{p}^{2+\alpha}}d^{4}%
x+\int_{||x||_{p}>1}\frac{\varphi(x)}{|f\left(  x\right)  |_{p}^{2+\alpha}%
}d^{4}x\right\}  -\varphi(0)\dfrac{1-p^{\alpha}}{1-p^{-\alpha-2}}\\
\times\int_{||x||_{p}>1}\frac{1}{|f\left(  x\right)  |_{p}^{2+\alpha}}d^{4}x\\
=\dfrac{1-p^{\alpha}}{1-p^{-\alpha-2}}\left\{  \int_{||x||_{p}\leq1}%
\frac{\varphi(x)-\varphi(0)}{|f\left(  x\right)  |_{p}^{2+\alpha}}d^{4}%
x+\int_{||x||_{p}>1}\frac{\varphi(x)}{|f\left(  x\right)  |_{p}^{2+\alpha}%
}d^{4}x\right\} \\
-\varphi(0)\dfrac{1-p^{\alpha}}{1-p^{-\alpha-2}}\frac{p^{-2\alpha}%
(1-p^{-2})\left(  1+p^{\alpha}\right)  }{1-p^{-2\alpha}}\\
=\left\langle K_{-\alpha},\varphi\right\rangle ,
\end{multline*}
where we used Lemmas \ref{lemma4} and \ref{lemma3}.

(iii)-(iv) We recall that if $\varphi\in\mathbf{S}\left(  \mathbb{Q}_{p}%
^{4}\right)  $, then $(K_{\alpha}\ast\varphi)(x)=\left\langle K_{\alpha
}(y),\varphi(x-y)\right\rangle $, and since $K_{\alpha}(-y)=K_{\alpha}(y)$, we
have $(K_{\alpha}\ast\varphi)(x)=\left\langle K_{\alpha}(y),\varphi
(x+y)\right\rangle $. Therefore (iii) follows from (i) and (iv) follows from (ii).
\end{proof}

\subsection{Lizorkin spaces of second kind}

Consider the spaces
\[
\mathbf{\Psi}:=\mathbf{\Psi}(\mathbb{Q}_{p}^{n})=\{\psi\in\boldsymbol{S}%
(\mathbb{Q}_{p}^{n})\mid\psi(0)=0\}
\]
and
\[
\mathbf{\Phi}:=\mathbf{\Phi}(\mathbb{Q}_{p}^{n})=\{\phi\mid\phi=\mathcal{F}%
[\psi],\ \ \mathbb{\psi}\in\mathbf{\Psi}(\mathbb{Q}_{p}^{n})\}.
\]

The space $\mathbf{\Phi}$\ is called \textit{the }$p$-adic \textit{Lizorkin
space of test functions of second kind}. We equip $\mathbf{\Psi}$ and
$\mathbf{\Phi}$ with the topology inherited from $\boldsymbol{S}%
(\mathbb{Q}_{p}^{n})$. Note that $\mathcal{F}:\mathbf{\Psi\rightarrow\Phi}$ is
an isomorphism of linear spaces and $\mathcal{F}\left(  \mathcal{F}\left[
\mathbf{\Psi}\right]  \right)  =\mathbf{\Psi}$.

Let $\mathbf{\Phi^{\prime}}=\mathbf{\Phi^{\prime}}(\mathbb{Q}_{p}^{n})$ denote
the topological dual of the space $\mathbf{\Phi}(\mathbb{Q}_{p}^{n})$. This is
space of\textit{ the }$p$\textit{-adic Lizorkin space of distributions of the
second kind}.

We define the Fourier transform of distributions $J\in\mathbf{\Phi}^{\prime
}(\mathbb{Q}_{p}^{n})$ and $G\in\mathbf{\Psi}^{\prime}(\mathbb{Q}_{p}^{n})$
by
\begin{align*}
\left\langle \mathcal{F}\left[  J\right]  ,\psi\right\rangle  &  =\left\langle
J,\mathcal{F}\left[  \psi\right]  \right\rangle ,\qquad\text{for any }\psi
\in\mathbf{\Psi}(\mathbb{Q}_{p}^{n}),\\
\left\langle \mathcal{F}\left[  G\right]  ,\phi\right\rangle  &  =\left\langle
G,\mathcal{F}\left[  \phi\right]  \right\rangle ,\qquad\text{\ for any }%
\phi\in\mathbf{\Phi}(\mathbb{Q}_{p}^{n}).
\end{align*}
It is clear that a $\mathcal{F}[\mathbf{\Psi}^{\prime}(\mathbb{Q}_{p}%
^{n})]=\mathbf{\Phi}^{\prime}(\mathbb{Q}_{p}^{n})$ and $\mathcal{F}%
[\mathbf{\Phi}^{\prime}(\mathbb{Q}_{p}^{n})]=\mathbf{\Psi}^{\prime}%
(\mathbb{Q}_{p}^{n})$. For further details about $p$-adic Lizorkin spaces the
reader may consult \cite{A-K-S}.

\subsection{The Riesz kernels form an Abelian group}

The goal of this section is to prove the following result:

\begin{theorem}
\label{mainA}For $\alpha,\beta\in\mathbb{C}$, $K_{\alpha}\ast K_{\beta
}=K_{\alpha+\beta}$ in $\mathbf{\Phi}^{\prime}(\mathbb{Q}_{p}^{4})$.
\end{theorem}

Before giving the proof we need to establish several auxiliary results.

\begin{definition}
Set $\ f^{\mathbf{\circ}}(x):=apx_{1}^{2}-px_{2}^{2}-ax_{3}^{2}+x_{4}^{2}$.
The Riesz kernel attached to $f^{\mathbf{\circ}}(x)$ is the distribution%
\[
K_{-\alpha}^{\mathbf{\circ}}(x):=|f^{\mathbf{\circ}}(x)|_{p}^{-\alpha}\text{
in }\mathbf{\Psi}^{\prime}(\mathbb{Q}_{p}^{4})\text{, for }\alpha\in
\mathbb{C}\text{.}%
\]

\end{definition}

\begin{proposition}
\label{fourietran} Considering $K_{\alpha}\in\mathbf{\Phi}^{\prime}%
(\mathbb{Q}_{p}^{4})$ and $K_{-\alpha}^{\mathbf{\circ}}\in\mathbf{\Psi
}^{\prime}(\mathbb{Q}_{p}^{4})$, we have%
\[
\mathcal{F}\left[  K_{\alpha}\right]  =K_{-\alpha}^{\mathbf{\circ}}\text{ for
}\alpha\not =2+\alpha_{k}\text{ and }\alpha\not =\alpha_{k},\text{ }%
k\in\mathbb{Z}\text{.}%
\]

\end{proposition}

\begin{proof}
The formula follows from Proposition \ref{eqfun4}.
\end{proof}

\begin{lemma}
\label{lemma5}%
\begin{equation}
\lim_{\alpha\rightarrow2+\alpha_{k}}\left\langle K_{\alpha},\varphi
\right\rangle =-\dfrac{1-p^{-2}}{\ln p}\left\langle \ln|f(x)|_{p}%
,\varphi(x)\right\rangle \text{ for }\varphi\in\mathbf{\Phi}(\mathbb{Q}%
_{p}^{4}). \label{K_alpha}%
\end{equation}

\end{lemma}

\begin{remark}
We understand the right-hand side in (\ref{K_alpha}) as the distribution
induced by the locally integrable function $\ln|f(x)|_{p}:\mathbb{Q}_{p}%
^{4}\setminus\{0\}\rightarrow\mathbb{R}$.
\end{remark}

\begin{proof}
Since%
\begin{multline*}
\lim_{\alpha\rightarrow2+\alpha_{k}}\left\langle K_{\alpha},\varphi
\right\rangle =\lim_{\alpha\rightarrow2+\alpha_{k}}\frac{1-p^{-\alpha}%
}{1-p^{\alpha-2}}%
{\displaystyle\int\limits_{\mathbb{Q}_{p}^{4}\smallsetminus\{0\}}}
|f(x)|_{p}^{\alpha-2}\varphi(x)d^{4}x\text{ }\\
=\lim_{\beta\rightarrow2}\left(  1-p^{-\beta}\right)
{\displaystyle\int\limits_{\mathbb{Q}_{p}^{4}\smallsetminus\{0\}}}
\left[  \frac{|x_{1}^{2}-ax_{2}^{2}-px_{3}^{2}+apx_{4}^{2}|_{p}^{\beta-2}%
-1}{1-p^{\beta-2}}\right]  \varphi(x)d^{4}x
\end{multline*}
by taking $\beta=\alpha-\alpha_{k}$ and by using the fact that $\int
\varphi(x)d^{4}x=0$. Now by passing to the limit under the integral sign we
have%
\[
\lim_{\alpha\rightarrow2+\alpha_{k}}\left\langle K_{\alpha},\varphi
\right\rangle =-(1-p^{-2})%
{\displaystyle\int\limits_{\mathbb{Q}_{p}^{4}\smallsetminus\{0\}}}
\frac{\ln|f(x)|_{p}}{\ln p}\varphi(x)d^{4}x.
\]
The passage to the limit under the integral sign is justified by the Lebesgue
Dominated Convergence Theorem and the inequality
\[
\left\vert \frac{e^{(\beta-2)\ln|f\left(  x\right)  |_{p}}-1}{1-e^{(\beta
-2)\ln p}}\right\vert \leq C\left\vert \frac{\ln|f(x)|_{p}}{\ln p}\right\vert
\text{ for }x\in\text{supp }\varphi\subset\mathbb{Q}_{p}^{4}\setminus
\{0\}\text{ and }|\beta-2|\leq1\text{,}%
\]
where $C=C(p,\text{supp }\varphi)$ is a positive constant.
\end{proof}

\begin{definition}
We define
\[
K_{2+\alpha_{k}}(x)=-\dfrac{1-p^{-2}}{\ln p}\ln|f\left(  x\right)  |_{p}%
\in\mathbf{\Phi^{\prime}}(\mathbb{Q}_{p}^{4}).
\]

\end{definition}

\begin{lemma}
\label{Lemma_Fourier_K2}%
\[
\left\langle \mathcal{F}\left[  K_{2+\alpha_{k}}\right]  ,\varphi\right\rangle
=\left\langle K_{-2}^{\mathbf{\circ}},\varphi\right\rangle ,\quad\text{for
}\varphi\in\mathbf{\Psi}(\mathbb{Q}_{p}^{4}).
\]

\end{lemma}

\begin{proof}
By using the fact that $K_{2+\alpha_{k}}=K_{2}$ and by Proposition
\ref{fourietran} we get
\[
\left\langle \mathcal{F}\left[  K_{2+\alpha_{k}}\right]  ,\varphi\right\rangle
=\lim_{\alpha\rightarrow2+\alpha_{k}}\left\langle K_{\alpha},\mathcal{F}%
\left[  \varphi\right]  \right\rangle =\lim_{\alpha\rightarrow2}\left\langle
K_{\alpha},\mathcal{F}\left[  \varphi\right]  \right\rangle =\lim
_{\alpha\rightarrow2}\left\langle K_{-\alpha}^{\mathbf{\circ}},\varphi
\right\rangle .
\]
By using $\varphi\left(  0\right)  =0$, we have
\[
\lim_{\alpha\rightarrow2}\left\langle K_{\alpha}^{\mathbf{\circ}}%
,\varphi\right\rangle =\lim_{\alpha\rightarrow2}%
{\displaystyle\int\limits_{\left\Vert x\right\Vert _{p}>p^{m}}}
\left\vert f^{\mathbf{\circ}}\left(  x\right)  \right\vert _{p}^{-\alpha
}\varphi\left(  x\right)  d^{4}x,
\]
where $m\in\mathbb{Z}$ is the exponent of local constancy of $\varphi$. We now
use \ the fact that $f^{\mathbf{\circ}}\left(  x\right)  $ is an elliptic
quadratic form to get
\begin{equation}
B\left\Vert x\right\Vert _{p}^{2}\leq|f^{\mathbf{\circ}}\left(  x\right)
|_{p}\leq A\left\Vert x\right\Vert _{p}^{2}\text{ for any }x\in\mathbb{Q}%
_{p}^{n}, \label{Zuniga_ineq2}%
\end{equation}
where $A$, $B$ are positive constants, cf. \cite[Lemma 1]{Z-G3}. Without loss
of generality we may assume that $B\leq1$ and that $m>0$. Then%

\[
\left\vert f^{\mathbf{\circ}}\left(  x\right)  \right\vert _{p}%
^{-\operatorname{Re}(\alpha)}\left\vert \varphi\left(  x\right)  \right\vert
\leq\frac{\left\vert \varphi\left(  x\right)  \right\vert }%
{B^{\operatorname{Re}(\alpha)}\left\Vert x\right\Vert _{p}^{2\operatorname{Re}%
(\alpha)}}\leq\frac{\left\vert \varphi\left(  x\right)  \right\vert
}{B^{2+\epsilon}p^{2m\left(  2-\epsilon\right)  }}%
\]
which is an integrable function on $\left(  \mathbb{Q}_{p}^{4}\smallsetminus
B_{m}(0)\right)  \cap$supp$\varphi$ and $\alpha\in\left(  2-\epsilon
,2+\epsilon\right)  $, where $\epsilon$ is a small fixed positive constant.
Therefore, by applying the Lebesgue Dominated Convergence Theorem,
\[
\lim_{\alpha\rightarrow2}\left\langle K_{\alpha}^{\mathbf{\circ}}%
,\varphi\right\rangle =%
{\displaystyle\int\limits_{\mathbb{Q}_{p}^{4}}}
\left\vert f^{\mathbf{\circ}}\left(  x\right)  \right\vert _{p}^{-2}%
\varphi\left(  x\right)  d^{4}x.
\]

\end{proof}

\begin{proposition}
\label{fourietransformtheorem} Considering $K_{\alpha}\in\mathbf{\Phi}%
^{\prime}(\mathbb{Q}_{p}^{4})$ and $K_{-\alpha}^{\mathbf{\circ}}%
\in\mathbf{\Psi}^{\prime}(\mathbb{Q}_{p}^{4})$,%
\[
\mathcal{F}\left[  K_{\alpha}\right]  =K_{-\alpha}^{\mathbf{\circ}}\text{ for
}\alpha\in\mathbb{C}\text{.}%
\]

\end{proposition}

\begin{proof}
The result follows from Proposition \ref{fourietran} and Lemma
\ref{Lemma_Fourier_K2}.
\end{proof}

\begin{lemma}
\label{lemma6} For any $\alpha$, $\beta\in\mathbb{C}$, $K_{-\alpha
}^{\mathbf{\circ}}\cdot K_{-\beta}^{\mathbf{\circ}}=K_{-(\alpha+\beta
)}^{\mathbf{\circ}}$ in $\in\mathbf{\Psi}^{\prime}(\mathbb{Q}_{p}^{4})$.
\end{lemma}

\begin{proof}
It follows from Lemma \ref{producdistributions}. Indeed, the functions
$K_{-\alpha}^{\mathbf{\circ}}$ and $K_{-\beta}^{\mathbf{\circ}}$ belong to
$L_{loc}^{1}$, and
\[
\lim_{\xi\rightarrow0}%
{\displaystyle\int\limits_{\mathbb{Q}_{p}^{4}}}
K_{-\alpha}^{\mathbf{\circ}}(x)\varphi(x)K_{-\beta}^{\mathbf{\circ}}%
(x-\xi)d^{4}x=%
{\displaystyle\int\limits_{\mathbb{Q}_{p}^{4}}}
K_{-\alpha}^{\mathbf{\circ}}(x)\varphi(x)K_{-\beta}^{\mathbf{\circ}}(x)d^{4}x.
\]
This last statement follows from the Lebesgue Dominated Convergence Theorem
and (\ref{Zuniga_ineq2}) by the inequality
\[
\left\vert K_{-\alpha}^{\mathbf{\circ}}(x)\varphi(x)K_{-\beta}^{\mathbf{\circ
}}(x-\xi)\right\vert \leq C\left(  \varphi,\alpha,\beta\right)  ||x||_{p}%
^{-2\operatorname{Re}\left(  \alpha\right)  -2\operatorname{Re}\left(
\beta\right)  }%
\]
for $x\in$supp $\varphi$ and $||\xi||_{p}\leq p^{m\left(  \varphi\right)  }$,
where $C\left(  \varphi,\alpha,\beta\right)  $ is a positive constant and
$m\left(  \varphi\right)  $ is the largest integer satisfying $\varphi
\mid_{B_{m\left(  \varphi\right)  }(0)}$ $\equiv0$.
\end{proof}

\subsubsection{Proof of Theorem \ref{mainA}.}

By Lemma \ref{lemma6}, for any $\alpha$, $\beta\in\mathbb{C}$, we have
\[
\mathcal{F}\left[  K_{-\alpha}^{\mathbf{\circ}}\cdot K_{-\beta}^{\mathbf{\circ
}}\right]  =\mathcal{F}\left[  K_{-\alpha}^{\mathbf{\circ}}\right]
\ast\mathcal{F}\left[  K_{-\beta}^{\mathbf{\circ}}\right]  =\mathcal{F}\left[
K_{-\left(  \alpha+\beta\right)  }^{\mathbf{\circ}}\right]  \text{ in
}\mathbf{\Phi}^{\prime}(\mathbb{Q}_{p}^{4}).
\]
By Proposition \ref{fourietransformtheorem}, $K_{\alpha}=\mathcal{F}\left[
K_{-\alpha}^{\mathbf{\circ}}\right]  $ for $\alpha\in\mathbb{C}$ since
$\mathcal{F}\left[  \mathcal{F}\left[  K_{\alpha}\right]  \right]  =K_{\alpha
}$, therefore for any $\alpha$, $\beta\in\mathbb{C}$, $K_{\alpha}\ast
K_{\beta}=K_{\alpha+\beta}$.

\begin{remark}
\label{nota_dim_2}The proof given Theorem \ref{mainA} can be extended to cover
the elliptic quadratic forms of dimension $2$, see Proposition \ref{eqfun2}.
\end{remark}

\section{\label{Section3}Pseudodifferential Operators and Fundamental
Solutions}

We take $f(\xi)=\xi_{1}^{2}-a\xi_{2}^{2}-p\xi_{3}^{2}+ap\xi_{4}^{2}$,
$f^{\ast}(\xi)=\frac{ap\xi_{1}^{2}-p\xi_{2}^{2}-a\xi_{3}^{2}+\xi_{4}^{2}}{ap}%
$, with $a\in\mathbb{Z}$ a quadratic non-residue module $p$, as in Section
\ref{Section2}. Given $\alpha>0$, we define the pseudodifferential operator
with symbol $\left\vert f^{\circ}\left(  \xi\right)  \right\vert _{p}^{\alpha
}$ by%
\[%
\begin{array}
[c]{lll}%
\mathbf{S}\left(  \mathbb{Q}_{p}^{4}\right)  & \rightarrow & C\left(
\mathbb{Q}_{p}^{4}\right)  \cap L^{2}\left(  \mathbb{Q}_{p}^{4}\right) \\
&  & \\
\varphi & \rightarrow & \left(  \boldsymbol{f}\left(  \partial,\alpha\right)
\varphi\right)  \left(  x\right)  :=\mathcal{F}_{\xi\rightarrow x}^{-1}\left(
\left\vert f^{\circ}\left(  \xi\right)  \right\vert _{p}^{\alpha}%
\mathcal{F}_{x\rightarrow\xi}\varphi\right)  .
\end{array}
\]
This operator is well-defined since $\left\vert f^{\circ}\left(  \xi\right)
\right\vert _{p}^{\alpha}\mathcal{F}_{x\rightarrow\xi}\varphi\in L^{1}\left(
\mathbb{Q}_{p}^{4}\right)  \cap L^{2}\left(  \mathbb{Q}_{p}^{4}\right)  $.
Note that $\left\vert pf^{\ast}\left(  \xi\right)  \right\vert _{p}^{\alpha
}=\left\vert f^{\circ}\left(  \xi\right)  \right\vert _{p}^{\alpha}$. Since
$\mathcal{F}^{-1}\left(  \left\vert f^{\circ}\right\vert _{p}^{\alpha}\right)
=K_{\alpha}$ (cf. Proposition \ref{fourietransformtheorem} ), by applying
Proposition \ref{Prop1} (iv), we get
\begin{equation}
\boldsymbol{f}\left(  \partial,\alpha\right)  \varphi=K_{-\alpha}\ast
\varphi=\dfrac{1-p^{\alpha}}{1-p^{-\alpha-2}}%
{\displaystyle\int\nolimits_{\mathbb{Q}_{p}^{4}}}
\frac{\varphi(x-y)-\varphi(x)}{|f(y)|_{p}^{\alpha+2}}d^{4}y, \label{Ext_oper}%
\end{equation}
for $\varphi\in\mathbf{S}\left(  \mathbb{Q}_{p}^{4}\right)  $.

Set $\mathcal{E}_{f,\alpha}\left(  \mathbb{Q}_{p}^{4}\right)  $ to be the
class consisting of locally constant functions $\varphi\left(  x\right)  $
satisfying%
\[%
{\displaystyle\int\limits_{\left\Vert x\right\Vert _{p}\geq p^{m}}}
\frac{\left\vert \varphi(x)\right\vert }{|f(x)|_{p}^{\alpha+2}}d^{4}%
x<\infty\text{ for some }m\in\mathbb{Z}\text{.}%
\]

The operator $\boldsymbol{f}\left(  \partial,\alpha\right)  $\ can be extended
to $\mathcal{E}_{f,\alpha}\left(  \mathbb{Q}_{p}^{4}\right)  $.

\begin{lemma}
\label{lemma6A}If $\varphi\in\mathcal{E}_{f,\alpha}\left(  \mathbb{Q}_{p}%
^{4}\right)  $, then the integral on the right- hand side of (\ref{Ext_oper}) converges.
\end{lemma}

\begin{proof}
Since $\varphi$ is locally constant there exists $l=l(x)\in\mathbb{Z}$ such
that $\varphi(x-y)-\varphi(x)=0$ for $\left\Vert y\right\Vert _{p}\leq p^{l}$,
thus, it is sufficient to show the convergence of the following integrals:
\[%
{\displaystyle\int\limits_{\left\Vert y\right\Vert _{p}>p^{l}}}
\frac{1}{|f(y)|_{p}^{\alpha+2}}d^{4}y,%
{\displaystyle\int\limits_{\left\Vert y\right\Vert _{p}>p^{l}}}
\frac{\left\vert \varphi(x-y)\right\vert }{|f(y)|_{p}^{\alpha+2}}d^{4}y=%
{\displaystyle\int\limits_{\left\Vert x-z\right\Vert _{p}>p^{l}}}
\frac{\left\vert \varphi(z)\right\vert }{|f(x-z)|_{p}^{\alpha+2}}d^{4}z.
\]
The convergence of the first integral follows from (\ref{Zuniga_ineq}). To
establish the convergence of the second integral, it is sufficient to \ show
the convergence of the integral
\begin{equation}%
{\displaystyle\int\limits_{\left\Vert x-z\right\Vert _{p}>p^{l}}}
\frac{\left\vert \varphi(z)\right\vert }{\left\Vert x-z\right\Vert
_{p}^{2\alpha+4}}d^{4}z, \label{int}%
\end{equation}
cf. (\ref{Zuniga_ineq}). The convergence of this last integral is established
by considering the cases: (i) $\left\Vert x\right\Vert _{p}<\left\Vert
z\right\Vert _{p}$, (ii) $\left\Vert x\right\Vert _{p}>\left\Vert z\right\Vert
_{p}$, (iii) $\left\Vert x\right\Vert _{p}=\left\Vert z\right\Vert _{p}$. The
verification of cases (i)-(ii) is left to the reader. In the case (iii), we
change variables as $x=p^{M}\widetilde{x}$, $z=p^{M}\widetilde{z}$ with
$\left\Vert \widetilde{x}\right\Vert _{p}=\left\Vert \widetilde{z}\right\Vert
_{p}=1$ in (\ref{int}), then%
\begin{align*}%
{\displaystyle\int\limits_{\left\Vert x-z\right\Vert _{p}>p^{l}}}
\frac{\left\vert \varphi(z)\right\vert }{\left\Vert x-z\right\Vert
_{p}^{2\alpha+4}}d^{4}z  &  =p^{2M\alpha}%
{\displaystyle\int\limits_{\substack{\left\Vert \widetilde{x}-\widetilde
{z}\right\Vert _{p}>p^{l+M}\\\left\Vert \widetilde{z}\right\Vert =1}}}
\frac{\left\vert \varphi(p^{M}\widetilde{z})\right\vert }{\left\Vert
\widetilde{x}-\widetilde{z}\right\Vert _{p}^{2\alpha+4}}d^{4}\widetilde{z}\\
&  \leq p^{2M\alpha-\left(  2\alpha+4\right)  \left(  l+M\right)  }%
{\displaystyle\int\limits_{\left\Vert \widetilde{z}\right\Vert =1}}
\left\vert \varphi(p^{M}\widetilde{z})\right\vert d^{4}\widetilde{z}<\infty.
\end{align*}

\end{proof}

The space of test functions $\mathbf{S}\left(  \mathbb{Q}_{p}^{4}\right)  $ is
not invariant under the action of $\boldsymbol{f}\left(  \partial
,\alpha\right)  $. But if we replace $\mathbf{S}\left(  \mathbb{Q}_{p}%
^{4}\right)  $ by $\mathbf{\Phi}\left(  \mathbb{Q}_{p}^{4}\right)  $ then
$\boldsymbol{f}\left(  \partial,\alpha\right)  \mathbf{\Phi}\left(
\mathbb{Q}_{p}^{4}\right)  =\mathbf{\Phi}\left(  \mathbb{Q}_{p}^{4}\right)  $.
The verification of this fact involves the same ideas used in the verification
of the corresponding assertion for the Taibleson operator, see e.g.
\cite[Lemma 9.2.5]{A-K-S}. On the other hand, the mapping%
\[%
\begin{array}
[c]{ccc}%
\mathbf{\Phi}^{\prime}\left(  \mathbb{Q}_{p}^{4}\right)  & \rightarrow &
\mathbf{\Phi^{\prime}}(\mathbb{Q}_{p}^{4})\\
&  & \\
J & \rightarrow & \boldsymbol{f}\left(  \partial,\alpha\right)  J:=\mathcal{F}%
^{-1}\left[  \left\vert f^{\circ}\right\vert _{p}^{\alpha}\mathcal{F}\left[
J\right]  \right]
\end{array}
\]
is a homeomorphism. This is a consequence of the fact that the map%
\[%
\begin{array}
[c]{lll}%
\mathbf{\Psi}\left(  \mathbb{Q}_{p}^{4}\right)  & \rightarrow & \mathbf{\Psi
}\left(  \mathbb{Q}_{p}^{4}\right) \\
&  & \\
\varphi & \rightarrow & \left\vert pf^{\ast}\right\vert _{p}^{\alpha}\varphi
\end{array}
\]
is a homeomorphism.

\begin{lemma}
\label{lemma7}The following formulas hold:

(i) $\left\langle \boldsymbol{f}\left(  \partial,\alpha\right)  J,\varphi
\right\rangle =\left\langle J,\boldsymbol{f}\left(  \partial,\alpha\right)
\varphi\right\rangle $ \ for any $J\in\mathbf{\Phi^{\prime}}(\mathbb{Q}%
_{p}^{4})$ and $\varphi\in\mathbf{\Phi}(\mathbb{Q}_{p}^{4})$\thinspace;

(ii) $\boldsymbol{f}\left(  \partial,\alpha\right)  J=K_{-\alpha}\ast J$ \ for
any $J\in\mathbf{\Phi^{\prime}}(\mathbb{Q}_{p}^{4})$.
\end{lemma}

\begin{proof}
(i) The formula follows from the following calculation:%
\begin{multline*}
\left\langle \boldsymbol{f}\left(  \partial,\alpha\right)  J,\varphi
\right\rangle =\left\langle \mathcal{F}^{-1}\left[  \left\vert f^{\circ
}\right\vert _{p}^{\alpha}\mathcal{F}\left[  J\right]  \right]  ,\varphi
\right\rangle =\left\langle J,\mathcal{F}\left[  \left\vert f^{\circ
}\right\vert _{p}^{\alpha}\mathcal{F}^{-1}\left[  \varphi\right]  \right]
\right\rangle \\
=\left\langle J,\mathcal{F}\left[  \left\vert f^{\circ}\left(  -\xi\right)
\right\vert _{p}^{\alpha}\mathcal{F}\left[  \varphi\right]  \left(
-\xi\right)  \right]  \right\rangle =\left\langle J,\mathcal{F}^{-1}\left[
\left\vert f^{\circ}\left(  \xi\right)  \right\vert _{p}^{\alpha}%
\mathcal{F}\left[  \varphi\right]  \left(  \xi\right)  \right]  \right\rangle
\\
=\left\langle J,\boldsymbol{f}\left(  \partial,\alpha\right)  \varphi
\right\rangle .
\end{multline*}

(ii) The formula follows from the fact that $\left\vert f^{\circ}\right\vert
_{p}^{\alpha}\mathcal{F}\left[  J\right]  \in\mathbf{\Psi^{\prime}}%
(\mathbb{Q}_{p}^{4})$ by using Proposition \ref{fourietransformtheorem}.
\end{proof}

\begin{definition}
\label{Classical_sol}Consider $\boldsymbol{f}\left(  \partial,\alpha\right)
:$ $\mathbf{\Phi}(\mathbb{Q}_{p}^{4})\rightarrow\mathbf{\Phi}(\mathbb{Q}%
_{p}^{4})$, and the equation%
\begin{equation}
\boldsymbol{f}\left(  \partial,\alpha\right)  u=\varphi\text{, }\quad
\varphi\in\mathbf{\Phi}(\mathbb{Q}_{p}^{4})\text{.} \label{Equation}%
\end{equation}
A classical \ solution of (\ref{Equation}) is a function $u$ belonging to the
domain of $\boldsymbol{f}\left(  \partial,\alpha\right)  $ which satisfies the
equation. A fundamental solution of (\ref{Equation}) is a distribution
$E_{\alpha}\in\mathbf{\Phi}^{\prime}\left(  \mathbb{Q}_{p}^{4}\right)  $ such
that $u\left(  x\right)  =\left(  E_{\alpha}\ast\varphi\right)  \left(
x\right)  $ is a classical solution of (\ref{Equation}) for any $\varphi
\in\mathbf{\Phi}(\mathbb{Q}_{p}^{4})$.
\end{definition}

\begin{lemma}
\label{lemma8}The following two assertions are equivalent:

(i) $E_{\alpha}\in\mathbf{\Phi}^{\prime}\left(  \mathbb{Q}_{p}^{4}\right)  $
is a fundamental solution of (\ref{Equation});

(ii) $\boldsymbol{f}\left(  \partial,\alpha\right)  E_{\alpha}=\delta$ in
$\mathbf{\Phi}^{\prime}\left(  \mathbb{Q}_{p}^{4}\right)  $.
\end{lemma}

\begin{proof}
(i) $\Leftrightarrow\boldsymbol{f}\left(  \partial,\alpha\right)  \left(
E_{\alpha}\ast\varphi\right)  =\varphi$ for any $\varphi\in\mathbf{\Phi
}(\mathbb{Q}_{p}^{4})\Leftrightarrow\left\vert f^{\circ}\right\vert
_{p}^{\alpha}\mathcal{F}\left[  E_{\alpha}\right]  =1$ in $\mathbf{\Psi
}^{\prime}\left(  \mathbb{Q}_{p}^{4}\right)  \Leftrightarrow\boldsymbol{f}%
\left(  \partial,\alpha\right)  E_{\alpha}=\delta$ in $\mathbf{\Phi}^{\prime
}\left(  \mathbb{Q}_{p}^{4}\right)  $.
\end{proof}

\begin{theorem}
\label{mainB}(i) The function%
\[
E_{\alpha}\left(  x\right)  =\left\{
\begin{array}
[c]{lll}%
\dfrac{1-p^{-\alpha}}{1-p^{\alpha-2}}|f(x)|_{p}^{\alpha-2} & \text{if} &
\alpha\neq2\\
&  & \\
-\dfrac{1-p^{-2}}{\ln p}\ln|f\left(  x\right)  |_{p} & \text{if} & \alpha=2
\end{array}
\right.
\]
is a fundamental solution of (\ref{Equation}).

(ii) Consider $\left\vert f\right\vert _{p}^{s}\in\mathbf{\Phi}^{\prime
}\left(  \mathbb{Q}_{p}^{4}\right)  $, $s\in\mathbb{C}$. Then$\ $%
\[
\boldsymbol{f}\left(  \partial,1\right)  \left\vert f\right\vert _{p}%
^{s+1}=\frac{\left(  1-p^{s+1}\right)  \left(  1-p^{-s-2}\right)  }{\left(
1-p^{-s-3}\right)  \left(  1-p^{s}\right)  }\left\vert f\right\vert _{p}%
^{s}\text{ in \ }\mathbf{\Phi}^{\prime}\left(  \mathbb{Q}_{p}^{4}\right)  .
\]

Here we are identifying $\left\vert f\right\vert _{p}^{s}$ with its
meromorphic continuation.
\end{theorem}

\begin{proof}
(i) By Lemma \ref{lemma8}, we have to show the existence of a distribution
$E_{\alpha}$ in $\mathbf{\Phi}^{\prime}\left(  \mathbb{Q}_{p}^{4}\right)  $
satisfying $\boldsymbol{f}\left(  \partial,\alpha\right)  E_{\alpha}=\delta$,
which is equivalent (by Lemma \ref{lemma7}\ (ii)) to solve $K_{-\alpha}\ast
E_{\alpha}=\delta$. By Theorem \ref{mainA} this equation has unique solution
$E_{\alpha}=K_{\alpha}$. Finally $u=E_{\alpha}\ast\varphi=\mathcal{F}%
^{-1}\left(  \frac{\mathcal{F}\left(  \varphi\right)  }{\left\vert f^{\circ
}\right\vert _{p}^{\alpha}}\right)  \in\mathbf{\Phi}(\mathbb{Q}_{p}^{4})$ for
$\varphi\in\mathbf{\Phi}(\mathbb{Q}_{p}^{4})$.

(ii) Note that
\[
\left\vert f\right\vert _{p}^{s+1}=\frac{1-p^{s+1}}{1-p^{-s-3}}K_{s+3}\text{
\ in \ }\mathbf{\Phi}^{\prime}\left(  \mathbb{Q}_{p}^{4}\right)  \text{ for
}s\notin\left\{  -3+\alpha_{k}\right\}  \cup\left\{  -1+\alpha_{k}\right\}  .
\]
Then by Lemma \ref{lemma7}\ (ii) and Theorem \ref{mainA}%
\begin{multline*}
\boldsymbol{f}\left(  \partial,1\right)  \left\vert f\right\vert _{p}%
^{s+1}=K_{-1}\ast\left\vert f\right\vert _{p}^{s+1}=\left(  \frac{1-p^{s+1}%
}{1-p^{-s-3}}\right)  K_{-1}\ast K_{s+3}=\left(  \frac{1-p^{s+1}}{1-p^{-s-3}%
}\right)  K_{s+2}\\
=\frac{\left(  1-p^{s+1}\right)  \left(  1-p^{-s-2}\right)  }{\left(
1-p^{-s-3}\right)  \left(  1-p^{s}\right)  }\left\vert f\right\vert _{p}%
^{s}\text{ }%
\end{multline*}
in \ $\mathbf{\Phi}^{\prime}\left(  \mathbb{Q}_{p}^{4}\right)  $ for
$s\notin\left\{  -1+\alpha_{k}\right\}  \cup\left\{  -2+\alpha_{k}\right\}
\cup\left\{  -3+\alpha_{k}\right\}  \cup\left\{  \alpha_{k}\right\}  $. The
announced formula follows by analytic continuation, since the distributions
$\boldsymbol{f}\left(  \partial,1\right)  \left\vert f\right\vert _{p}^{s+1}$
and $\frac{\left(  1-p^{s+1}\right)  \left(  1-p^{-s-2}\right)  }{\left(
1-p^{-s-3}\right)  \left(  1-p^{s}\right)  }\left\vert f\right\vert _{p}^{s}$
agree on an open and connected subset of the complex plane.
\end{proof}

\begin{remark}
\label{nota_ultima}Similar results are valid for pseudodifferential operators
attached to elliptic quadratic forms of dimension $2$.
\end{remark}

\end{document}